\documentclass[EJP]{ejpecp} 

\usepackage[T1]{fontenc}
\usepackage[latin1]{inputenc}

\SHORTTITLE{Seneta-Heyde norming for branching random walks} 

\TITLE{A revisited proof of the Seneta-Heyde norming for \\branching random walks under optimal assumptions\thanks{PB and PM partially supported by ANR Liouville (ANR-15-CE40-0013) and ANR GRAAL (ANR-14-CE25-0014). PM partially supported by a CRM Simons Research Fellowship.}}

\AUTHORS{%
  Pierre~Boutaud\footnote{Laboratoire de Mathématiques d'Orsay, Univ.Paris-Sud, CNRS, Université Paris-Saclay, 91405 Orsay, France.
    E-mail: \texttt{pierre.boutaud at u-psud dot fr}} 
  \and 
  Pascal~Maillard\footnote{Institut de Mathématiques de Toulouse, CNRS, UMR5219, Université de Toulouse, 118 route de Narbonne, F-31062 Toulouse cedex 09, France. E-mail: \texttt{Pascal.Maillard at math dot univ-toulouse dot fr}}}
	

\AMSSUBJ{60J80}
\AMSSUBJSECONDARY{60J50 ; 60B10}

\SUBMITTED{February 20, 2019} 
\ACCEPTED{August 5, 2019} 

\ARXIVID{1902.05330v2} 

\VOLUME{0}
\YEAR{2012}
\PAPERNUM{0}
\DOI{vVOL-PID}

\ABSTRACT{We introduce a set of tools which simplify and streamline the proofs of limit theorems concerning near-critical particles in branching random walks under optimal assumptions. We exemplify our method by giving another proof of the Seneta-Heyde norming for the critical additive martingale, initially due to Aïdékon and Shi. The method involves in particular the replacement of certain second moment estimates by truncated first moment bounds, and the replacement of ballot-type theorems for random walks by estimates coming from an explicit expression for the potential kernel of random walks killed below the origin. Of independent interest might be a short, self-contained proof of this expression, as well as a criterion for convergence in probability of non-negative random variables in terms of conditional Laplace transforms.}

\newcommand{\dg}{\delta}
\newcommand{\eg}{\varepsilon}
\renewcommand{\lg}{\lambda}

\renewcommand{\AA}{\mathcal{A}}
\newcommand{\FF}{\mathcal{F}}
\newcommand{\GG}{\mathcal{G}}
\newcommand{\UU}{\mathcal{U}}

\begin{document}



\section{Introduction}\label{secIntro}

In the theory of branching processes, many limit theorems hold under so-called $L \log L$-type moment conditions which are both sufficient and necessary. The most famous and historically first is the Kesten--Stigum theorem \cite{Kesten1966}, which states in particular that a supercritical Galton--Watson process $(Z_n)_{n\ge0}$ grows asymptotically like $Wm^n$ as $n\to\infty$, where $m$ is the mean of the offspring distribution and $W$ is a random variable which is non-degenerate if and only if $\mathbb{E}[L \log L] < \infty$, where $L$ is a random variable equal in law to the number of offspring of an individual. In fact, $W$ is the limit of the martingale $W_n = m^{-n}Z_n$ and another statement of the theorem says that the martingale $(W_n)_{n\ge0}$ is uniformly integrable if and only if $\mathbb{E}[L\log L]<\infty$.

In the context of branching random walks, Lyons \cite{Lyons1998} has shown an analogous theorem for the so-called \emph{additive martingales} arising naturally in this context. His theorem pertains mostly to those additive martingales whose parameter is in the so-called subcritical, or, using statistical physics terminology, high-temperature regime. The martingales in this regime describe the asymptotic growth of the particles in the \emph{bulk}, i.e.~in regions where the number of particles grows exponentially with time. In contrast, the last decade has seen considerable interest in the \emph{extremal particles} in branching random walks, as well as in related models such as the two-dimensional Gaussian Free Field, Gaussian multiplicative chaos and characteristic polynomials of certain random matrices, see e.g. \cite{ShiLectureNotes,ZeitouniLNBRW,RhodesVargasReview,BovierBook} for fairly recent reviews. In these models and in the branching random walk in particular, it is well-known that the asymptotics of the \emph{extremal} or \emph{near-extremal} particles are strongly related to the so-called \emph{derivative martingale}, which is the derivative of the additive martingale with respect to its parameter at its critical value. It is therefore natural to ask for sufficient and necessary $L\log L$-type conditions for the convergence of the derivative martingale to a non-degenerate limit. Such a condition, together with proof of sufficiency, has been given by Aïdékon \cite{Aidekon2013}, with necessity subsequently established by Chen \cite{Chen2015}. We will refer to it as \emph{Aïdékon's condition}.

Aïdékon's condition arises generically in limit theorems concerning critical or near-critical particles in branching random walk. A prime example is the convergence in law of the recentered minimum \cite{Aidekon2013}. Another important example is the so-called \emph{Seneta-Heyde} norming of the additive martingale at critical parameter: it has been shown by Aïdékon and Shi \cite{Aidekon2014} that this martingale, properly renormalized, converges in probability to the same limit as the derivative martingale, under Aïdékon's condition. Their proof has been adapted by He, Liu and Zhang \cite{He2016} to cases where a certain variance $\sigma^2$ (defined in Equation \eqref{eq:variancefinie} below) is infinite and by Aru, Powell and Sepúlveda \cite{Aru2017} to an analogous result for Gaussian multiplicative chaos. The proofs of such limit theorems are often quite involved and technical. At their heart lies the so-called \emph{spine decomposition} introduced by Lyons, Pemantle and Peres \cite{Lyons1995} for Galton--Watson processes and adapted by Lyons \cite{Lyons1998} to the branching random walk. However, in order to cope with the extremal or near-extremal particles, quite involved truncation techniques have been introduced. A self-contained treatment of these techniques appears in Shi \cite{ShiLectureNotes}. These include the following:
\begin{itemize}
\item Second moment estimates for quantities restricted to a certain subset of the particles and first moment bounds on the remainder by so-called \emph{peeling lemmas}
\item ballot-type theorems for random walks conditioned to stay above certain space-time curves.
\end{itemize}
We emphasize that these techniques are not only quite technical, but also require the ad-hoc construction of certain quantities and sets of particles specifically tailored to the problem at hand. Other techniques using $L^p$ estimates can be used, but they require more restrictive assumptions, see e.g.~Kyprianou and Madaule \cite{Kyprianou2015a}.

In the present article, we give a new proof of the Seneta-Heyde norming for the critical additive martingale in branching random walks, valid under optimal assumptions. We find this proof to be simpler and more streamlined than the original one by Aïdékon and Shi due to several technical improvements, amongst others:
\begin{itemize}
\item the second moment estimates and peeling lemmas are replaced by certain truncated \emph{first} moment estimates
\item the use of ballot-type theorems is replaced by other, softer methods, in particular bounds on the potential kernel of random walks killed below the origin.
\end{itemize}
We believe that our methods not only make the proof simpler, but that they are also more versatile in that they can be used as a general toolbox for proving limit theorems involving extremal or near-extremal particles of the branching random walk under optimal assumptions. In fact, the present article is part of a program that aims to establish limit theorems for branching random walks under non-standard assumptions and the tools developed here will be of use later in the program.

The methods from this article can of course be adapted for analogous continuous-time processes, such as certain branching L\'evy processes with possibly non-local branching. Local branching simplifies some arguments -- mostly in the proof of Lemma~\ref{lem:BigTerm} in Section~\ref{secGrosTerme}. In the special case of branching Brownian motion with local branching further simplifications arise, due to the fact that key quantities related to Brownian motion killed at 0 (harmonic function, potential kernel, survival probability, \ldots) admit simple explicit expressions.

\subsection*{Definitions and results}

We consider discrete-time real-valued branching random walks (BRWs), which can be informally described as follows. At time $n=0$, we start with one initial particle at the origin. Then, at each time step $n\ge 1$, every particle dies and gives birth to a random, possibly infinite number of particles distributed randomly on the real line. More precisely, the children of a particle at position $x\in\mathbb{R}$ are positioned at $x+X_1,x+X_2,\ldots$, where the vector $(X_1,X_2,\ldots)$ follows a given law $\Theta$, called the offspring distribution of the branching random walk. At each generation, the reproduction events are independent. Also, it is possible for several particles to share the same position. We further assume that the Galton-Watson process formed by the number of particles at each generation is super-critical, so that the system survives with positive probability.

Formally, the branching random walk can be constructed as a stochastic process indexed by the Ulam-Harris tree $\mathcal{U}= \bigcup_{n\ge0} (\mathbb{N}^*)^n$, where $\mathbb{N}^* = \{1,2,\ldots\}$. Particles are identified with vertices $u\in\mathcal{U}$, i.e.~words over the alphabet $\mathbb{N}^*$. The length of the word $u$, i.e.~the generation of the particle, is denoted by $|u|$. The position of the particle $u$ is denoted by $X_u$. If the particle indexed by $u$ does not exist, we set $X_u = +\infty$. The branching random walk described above then defines a process $(X_u)_{u\in\mathcal{U}}$ taking values in $\bar{\mathbb{R}} = \mathbb{R}\cup\{+\infty\}$ and the offspring distribution $\Theta$ is a probability distribution on $(\bar{\mathbb{R}})^{\mathbb{N}^*}$. We further convene that mathematical expressions such as sums or products over the set $\{|u| = n\}$ of particles at generation $n$ are meant to ignore those $u$ for which $X_u = +\infty$.

As mentioned above, we assume that the branching is super-critical, i.e.
\[
\mathbb{E}\left[\sum_{|u|=1} 1\right] > 1.
\]
Furthermore, we work in the so-called boundary case, meaning we suppose that
\begin{equation}\label{eq:boundarycaseHYP}
\mathbb{E}\left[\sum_{|u|=1}e^{-X_{u}}\right]=1\quad\text{and}\quad\mathbb{E}\left[\sum_{|u|=1}X_{u}e^{-X_{u}}\right]=0.
\end{equation}
The second equality in \eqref{eq:boundarycaseHYP} implicitly assumes that the expectation is well-defined, which is automatically the case under the next assumption:
\begin{equation}
\sigma^2=\mathbb{E}\left[\sum_{|u|=1}X_{u}^{2}e^{-X_{u}}\right]\in(0,\infty)\label{eq:variancefinie}.
\end{equation}
Note that $\sigma^2<\infty$ holds for example if $\mathbb{E}[e^{-\theta X_u}]<\infty$ for $\theta$ in a neighborhood of 1, and $\sigma^2>0$ holds as soon as the $X_u$ are not all equal to 0 or $+\infty$, almost surely.

It is a well-known consequence of \eqref{eq:boundarycaseHYP} and the branching property that the processes $(W_n)_{n\ge0}$ and $(D_n)_{n\ge0}$, defined by
\begin{equation*}
W_{n}=\sum_{|u|=n}e^{-X_{u}},\qquad D_{n}=\sum_{|u|=n}X_{u}e^{-X_{u}},
\end{equation*}
are martingales with respect to the canonical filtration $(\FF_n)_{n\ge0}$ of the BRW, defined by $\FF_{n}=\sigma(X_{u},|u|\le n)$, see for exemple \cite{Biggins2004}. We will refer to $(W_n)_{n\ge0}$ as the \emph{additive martingale} or \emph{Biggins' martingale}, in reference to Biggins \cite{Biggins1977} and to $(D_n)_{n\ge0}$ as the \emph{derivative martingale}. The second equality in \eqref{eq:boundarycaseHYP} implies that $W_{n}$ converges almost surely to 0 \cite{Lyons1998}. In particular we have $\min_{|u|=n} X_{u}\longrightarrow \infty$ a.s., as $n\to\infty$. As for the derivative martingale, under assumptions \eqref{eq:boundarycaseHYP} and \eqref{eq:variancefinie}, Biggins and Kyprianou  \cite{Biggins2004} showed that $D_{n}$ converges a.s. to a finite nonnegative limit $D_\infty$.

We introduce the following moment conditions:
\begin{align}
\mathbb{E}\left[W_{1}\log_{+}^{2}W_{1}\right]&<\infty\label{eq:Wlog2W},\\
\mathbb{E}\left[Z_{1}\log_{+}Z_{1}\right]&<\infty,\quad \text{where } Z_{1}=\sum_{|u|=1}X_{u}^{+}e^{-X_u}\label{eq:DlogD}.
\end{align}
Here, and throughout the article, we use the notations $\log_{+}(x)=\log(x)\vee0$, $x^{+}=x\vee0$, where $x\vee y=\max(x,y)$ and $x\wedge y=\min(x,y)$. Under the additional assumptions \eqref{eq:Wlog2W} and \eqref{eq:DlogD}, Aïdékon \cite{Aidekon2013} proved that $D_\infty>0$ a.s.\ on the event of survival of the branching random walk. Later, Chen \cite{Chen2015} showed the converse result in the sense that if \eqref{eq:boundarycaseHYP} and \eqref{eq:variancefinie} hold then the limit is non-trivial if and only if conditions \eqref{eq:Wlog2W} and \eqref{eq:DlogD} hold.

As mentioned in the introduction, the main result of this paper is a new proof of the following result by Aïdékon and Shi \cite{Aidekon2014}. We believe this proof to be simpler and more streamlined and the tools established in proving it will be useful to work with in other settings. 

\begin{theorem}[Aïdékon, Shi \cite{Aidekon2014}]\label{th:mainresult}
Assume \eqref{eq:boundarycaseHYP}, \eqref{eq:variancefinie}, \eqref{eq:Wlog2W} and \eqref{eq:DlogD} hold. We have
\begin{equation}
\sqrt{n}W_{n}\underset{n\to\infty}{\longrightarrow}\sqrt{\frac{2}{\pi\sigma^2}}D_{\infty}\quad\text{in probability}.
\end{equation}
\end{theorem}

The remainder of the article is organized as follows.
Section \ref{secSpineRenewal} contains some preliminaries, namely the spinal decomposition and the many-to-one formula (Section~\ref{ssecSpine}) and some properties of a certain renewal function associated to this decomposition (Section~\ref{ssecRenewal}).
Section~\ref{secProofMain} contains the proof of Theorem \ref{th:mainresult}, as well as a comparison with previous proofs of the same result. In particular, we present in this section the new technical tools going into our proof. The proof of one key ingredient (Proposition \ref{prop:smallo}) is deferred to Section \ref{secProofProp}, where most of the work is done.
Appendix~\ref{appGreen} contains a formula for the potential kernel of random walks on a half-line in terms of associated renewal measures. Appendix~\ref{appLaplace} contains a criterion for convergence in probability of non-negative random variables using Laplace transforms. Appendix~\ref{appTauberian} contains a certain Tauberian-type lemma involving truncated first moments of a non-negative random variable.

\section{Spinal decomposition and renewal functions}\label{secSpineRenewal}
\subsection{The spinal decomposition}\label{ssecSpine}

In this section, we recall a change of measure and an associated spinal decomposition of the BRW due to Lyons~\cite{Lyons1998}. It will be helpful to allow the initial particle of the BRW to sit at an arbitrary position $x\in\mathbb{R}$, this will be denoted by adding the subscript $x$ as in $\mathbb{P}_x$ and $\mathbb{E}_x$ (if $x=0$, the subscript is ignored). Then $(W_n)_{n\ge0}$ is still a non-negative martingale with $W_0 = e^{-x}$. Define $\FF_{\infty}=\bigvee_{n\ge0}\FF_{n}$. Using Kolmogorov's extension theorem, for every $x\in\mathbb{R}$, there exists a probability measure $\mathbb{P}^*_{x}$ on $\FF_{\infty}$ such that for every generation $n\ge 0$,
\begin{equation}
\frac{d\mathbb{P}^{*}_{x}}{d\mathbb{P}_{x}}\Big|_{\FF_{n}}=e^{x}W_{n}.
\end{equation}

Following Lyons \cite{Lyons1998} we see $\mathbb{P}_x^*$ as the projection to $\FF_{\infty}$ of a probability (also denoted $\mathbb{P}_x^*$) defined on a bigger probability space equipped with a so-called \emph{spine}, a distinguished ray in the tree. We will denote the vertex on the spine at generation $n$ by $\xi_{n}$ and its position by $X_{\xi_n}$. The spinal BRW evolves as follows under $\mathbb{P}^*_x$:
\begin{itemize}
\item Start at generation 0 with one particle $\xi_0$ at position $x$.
\item At generation $n$, all particles except $\xi_{n}$ reproduce according to the point process $\Theta$ and $\xi_{n}$ reproduces according to the size-biased reproduction law $\Theta^*$ defined by
\[
\frac{d\Theta^*}{d\Theta}(x_1,x_2,\ldots) = \sum_{i\ge1} e^{-x_i}.
\]
\item The spine at generation $n+1$ is chosen amongst the children $u$ of $\xi_{n}$ with probability proportional to $e^{-X_{u}}$.
\end{itemize}

The following \emph{many-to-one formula} can be deduced from Lyons \cite{Lyons1998}, see also Aïdékon~\cite{Aidekon2013}.
\begin{proposition}[Many-to-one formula]\label{prop:MTOformula}
For any $x\ge 0$, $n\in\mathbb{N}=\{0,1,\dots\}$ and every uniformly bounded family $\left(H_{n}(u)\right)_{u\in\UU}$ of $\FF_{n}$-measurable random variables, one has
\begin{equation}
\mathbb{E}_{x}\left[\sum_{|u|=n}e^{-X_{u}}H_{n}(u)\right]=e^{-x}\mathbb{E}_{x}^{*}\left[H_{n}(\xi_{n})\right].
\end{equation}
\end{proposition}

The spinal decomposition implies that the process $(X_{\xi_{n}})_{n\in\mathbb{N}}$ follows the law of a random walk under $\mathbb{P}^*_x$ (whose increments do not depend on $x$). Furthermore, Proposition~\ref{prop:MTOformula} together with assumptions \eqref{eq:boundarycaseHYP} and \eqref{eq:variancefinie} shows that this random walk is centered and has finite positive variance:
\begin{equation}\label{eq:MomentsSpine}
\mathbb{E}^*[X_{\xi_{1}}]=0,\quad\mathbb{E}^*[X_{\xi_{1}}^{2}] = \sigma^2\in(0,\infty),
\end{equation}
where $\sigma^2$ is the same as in \eqref{eq:variancefinie}.
The many-to-one formula is a powerful tool which allows to express many quantities of the branching random walk in terms of the random walk $(X_{\xi_{n}})$.

\subsection{The renewal function \texorpdfstring{$R$}{R}}\label{ssecRenewal}

Throughout the article, we denote by $R$ the renewal function associated to the strictly descending ladder heights of the random walk $(X_{\xi_n})_{n\ge0}$, as defined in Appendix \ref{appGreen}. Explicitly, one may express $R$ by
\begin{equation*}
R(x)=\sum_{n\ge0}\mathbb{P}^*\left(X_{\xi_n}\ge-x, X_{\xi_n}<\min_{0\le k\le n}X_{\xi_k}\right).
\end{equation*}
Note that $R(0)=1$ and that for all $x<0$, $R(x)=0$. We recall the following fact (see Lemma~\ref{lem:R} in the appendix):
\begin{proposition}\label{prop:Rharmonic}
The renewal function $R$ is harmonic for the random walk killed when entering $(-\infty,0)$:
\begin{equation*}
\forall x\ge0, R(x)=\mathbb{E}_x^*\left[R(X_{\xi_1})\boldsymbol 1_{X_{\xi_1}\ge0}\right].
\end{equation*}
\end{proposition}

Recall from the previous section that the random walk $(X_{\xi_n})_{n\ge0}$ is centered and of finite variance $\sigma^2$. In particular, the strictly descending ladder heights have finite expectation, see e.g. Rogozin~\cite{Rogozin1964}.
The following lemma recalls some well-known quantitative results concerning the function $R$ and the probability that the random walk $(X_{\xi_n})_{n\ge0}$ stays non-negative in terms of $R$. 
\begin{lemma}\label{lem:Kozlov}
For all $x\ge0$, as $n\to\infty$, we have
\begin{equation}\label{eq:Koz1}
\mathbb{P}^*_x\left(\min_{k\le n} X_{\xi_k}\ge 0\right)\sim \frac{\theta R(x)}{\sqrt{n}},
\end{equation}
and for all $x\ge0$, $n\ge1$,
\begin{equation}\label{eq:Koz2}
\mathbb{P}^*_x\left(\min_{k\le n} X_{\xi_k}\ge 0\right)\le \frac{\theta' R(x)}{\sqrt{n}},
\end{equation}
where $\theta$ and $\theta'$ are positive constants. Furthermore,
\begin{equation}
\label{eq:Rlimit}
\frac{\theta R(x)}{x}\underset{x\to\infty}{\longrightarrow}\sqrt{\frac{2}{\pi\sigma^2}},
\end{equation}
where $\sigma^2$ is defined in \eqref{eq:variancefinie}.
\end{lemma}
Equations~\eqref{eq:Koz1} and \eqref{eq:Koz2} in Lemma~\ref{lem:Kozlov} are due to Kozlov \cite{Kozlov1976}. Equation~\eqref{eq:Rlimit} can be found for example in Aïdékon and Shi  \cite{Aidekon2014} and is derived there as a consequence of Feller's renewal theorem and Sparre~Andersen's identities for random walks (see for example \cite{Kersting2017}, section 4.2). A different approach, in the spirit of Madaule \cite{Madaule2016}, is to use an invariance principle for the Doob $R$-transform of the random walk killed below the origin, see Section~\ref{secProofProp} for a precise definition of this process. This allows to identify the constant $\sqrt{2/\pi}$ as the expectation of $1/Z_1$, where $Z_1$ is a 3-dimensional Bessel process at time 1, starting from 0.

Another consequence of Feller's renewal theorem \cite{Feller1971} is that for every $y\in\mathbb{R}$, the difference $R(x+y)-R(x)$ is uniformly bounded in $x$. In particular (this also follows from \eqref{eq:Rlimit} in Lemma~\ref{lem:Kozlov}), there exists a finite $c_1$ such that
\begin{equation}\label{eq:defc1}
\forall x\in\mathbb{R}: R(x)\le c_{1}(1+x^+).
\end{equation}

The following lemma is an easy consequence of \eqref{eq:defc1} and will be used in Section~\ref{secProofProp}.
\begin{lemma}\label{lem:BoundR}
For all $x,y\in\mathbb{R}$, we have
\begin{equation*}
R(x+y)\le c_1 (1+x^+)(1+y^+).
\end{equation*}
\end{lemma}


\section{Outline of the proof of Theorem \ref{th:mainresult}}\label{secProofMain}

We define the following quantities, for $n, k_0\ge 0$ :
\begin{align}
W_{n}'&=\sum_{|u|=n}e^{-X_{u}}\boldsymbol{1}_{\min_{v\le u}X_v\ge0}\label{eq:WnPrime}\\
W_{n,k_0}''&=\sum_{|u|=n}e^{-X_{u}}\boldsymbol{1}_{\min_{v\le u, |v|\ge k_0}X_v\ge0}
\end{align}
In other words, $W_n'$ and $W_{n,k_0}''$ are obtained from $W_n$ by removing from the sum the contribution of the particles going below the origin at some time $k\le n$ or $k_0 \le k\le n$, respectively.

Remember that $\min_{|u|=n}X_{u}\to\infty$ almost surely on the event of survival. Thus \begin{equation}\label{eq:primeprime}\forall\varepsilon>0,\exists k_{0}:\mathbb{P}(\forall n, W_{n,k_0}''=W_{n})>1-\varepsilon.\end{equation}

\begin{proposition}\label{prop:MomentKozlov}
Let $x\ge0$. Then, as $n\to\infty$,
\begin{equation*}
\mathbb{E}_{x}\left[W_{n}'\right]\sim \frac{\theta R(x)e^{-x}}{\sqrt{n}},
\end{equation*}
and for all $n\ge 0$,
\begin{equation*}
\mathbb{E}_{x}\left[W_{n}'\right]\le \frac{\theta'R(x)e^{-x}}{\sqrt{n+1}},
\end{equation*}
where $\theta$ and $\theta'$ are the constants from Lemma~\ref{lem:Kozlov}.
\end{proposition}

\begin{proof}
We have :
\begin{align*}
\mathbb{E}_{x}\left[W_{n}'\right]&=\mathbb{E}_{x}\left[\sum_{|u|=n}e^{-X_u}\boldsymbol{1}_{\min_{v\le u}X_v\ge0}\right]\\
&=e^{-x}\mathbb{P}_{x}^{*}\left(\min_{i\le n}X_{\xi_i}\ge 0\right),\ \text{by the many-to-one formula}.
\end{align*}
Using Lemma~\ref{lem:Kozlov} ends the proof.
\end{proof}

In addition to the first moment estimate, we will need to show that the quantity $\sqrt{n}W_{n}'$ concentrates sufficiently well around its expectation when $x$ is large. Typically, one decomposes $W_n'$ into a sum of two terms, involving ``good'' and ``bad'' particles, respectively. One then calculates the second moments of the first term and bounds the first moment of the second term by a so-called \emph{peeling lemma}, tailored to the problem at hand (see e.g.~Theorem 5.14 in \cite{ShiLectureNotes}). One key novel idea from this article, which greatly simplifies the calculations, is to replace this by estimates of truncated first moments. With these, one can then obtain precise bounds on conditional Laplace transforms which yield the convergence in probability by virtue of Lemma~\ref{lem:AnnexB} in the appendix\footnote{This idea has already appeared in \cite[Theorem~B.1]{slowdown} without the justification provided by Lemma~\ref{lem:AnnexB}.}.

The first moment estimate we will use is the following:
\begin{proposition}\label{prop:smallo}
For every $\varepsilon>0$, there exists a positive function $h$,  such that $\frac{h(x)}{R(x)}\to 0$ as $x\to\infty$ and such that the following holds: for every $x\ge0$, we have
\begin{equation*}
\limsup_{n\to\infty} \mathbb{E}_{x}\left[\sqrt{n}W_{n}'\boldsymbol{1}_{\sqrt{n}W_{n}'\ge\varepsilon}\right]\le  h(x)e^{-x}.
\end{equation*}
\end{proposition}
Proposition \ref{prop:smallo} will be proven in the next section. Its proof relies on the spinal decomposition from Section \ref{ssecSpine} as well as on two ingredients the use of which we believe to be new in this context : 
\begin{itemize}
\item a lemma by Kersting and Vatutin \cite{Kersting2017} on the convergence of functionals of random walks conditioned to stay above the origin until a finite time $n$ to analogous functionals of random walks conditioned to stay above the origin for all time.
\item an explicit formula for the potential kernel of random walks killed below the origin.
\end{itemize}

We now have all the tools we need in order to prove Theorem \ref{th:mainresult}.

\begin{proof}[Proof of Theorem \ref{th:mainresult}]
~\\
We start by giving the outline of the proof.

We first show that for any $\lg>0$, $\mathbb{E}\left[\exp\left(-\lg \sqrt{n}W_{n,k_0}''\right)\Big{|}\FF_{k_0}\right]$ converges in probability to $\exp\left(-\lambda\sqrt{\frac{2}{\pi\sigma^2}}D_{\infty}\right)$, as first $n$, then $k_{0}$, tend to infinity. To do so we prove a lower and an upper bound, the latter relying crucially on Proposition~\ref{prop:smallo}. By Lemma~\ref{lem:AnnexB} in the appendix and a Cantor diagonal extraction argument, this yields convergence in probability of $\sqrt{n}W_{n,k_{0}(n)}''$ to $\sqrt{\frac{2}{\pi\sigma^2}}D_\infty$. We then use \eqref{eq:primeprime} to conclude the proof.

We now get to the details. First we show a lower bound on the conditional Laplace transform.
For any $\lambda>0$, we have
\begin{align}
\mathbb{E}\left[\exp\left(-\lg \sqrt{n}W_{n,k_0}''\right)\Big{|}\FF_{k_0}\right]&=\prod_{|u|=k_0} \mathbb{E}_{X_{u}}\left[\exp\left(-\lg\sqrt{n}W_{n-k_0}'\right)\right]\nonumber\\
&\ge \exp\left(-\lg\sqrt{n}\sum_{|u|=k_0}\mathbb{E}_{X_u}\left[W_{n-k_0}'\right]\right),\ \text{by Jensen's inequality}.\label{eq:Jensen}
\end{align}
By Proposition~\ref{prop:MomentKozlov}, for every $x\in\mathbb{R}$, $\sqrt{n}\mathbb{E}_{x}[W_{n-k_0}']$ converges to $\theta R(x)e^{-x}$ as $n\to\infty$ and is bounded from above by $\theta' R(x)e^{-x}$, with $\theta$ and $\theta'$ as in the statement of that proposition. Furthermore, using Proposition~\ref{prop:MTOformula}, one easily checks that $\sum_{|u|=k_0}R(X_u)e^{-X_u}$ is finite in expectation and therefore almost surely. By dominated convergence, we get almost surely
\begin{equation}\label{eq:converge}
\sqrt{n}\sum_{|u|=k_0}\mathbb{E}_{X_u}\left[W_{n-k_0}'\right]\underset{n\to\infty}{\longrightarrow}\sum_{|u|=k_0}\theta R(X_u)e^{-X_u}.
\end{equation}
By Equations \eqref{eq:Jensen} and \eqref{eq:converge}, we get almost surely,
\begin{equation}\label{eq:lower}
\underset{n\to\infty}{\liminf}\ \mathbb{E}\left[\exp\left(-\lg \sqrt{n}W_{n,k_0}''\right)\Big{|}\FF_{k_0}\right]\ge \exp\left(-\lg\sum_{|u|=k_0}\theta R(X_u)e^{-X_u}\right).
\end{equation}

We now deal with the upper bound. We notice that for $\lambda>0$ fixed and for any $\lg'\in(0,\lg)$, there exists $\varepsilon>0$ such that 
\begin{equation}\label{eq:deltabound}
\forall x\in[0,\varepsilon), e^{-\lambda x}\le1-\lg' x.
\end{equation}
Fix $\lg>0$ and $\lg'\in(0,\lg)$ (that will later tend to $\lg$), and take $\varepsilon$ satisfying \eqref{eq:deltabound}. We compute
\begin{align*}
\mathbb{E}\left[\exp\left(-\lg \sqrt{n}W_{n,k_0}''\right)\Big{|}\FF_{k_0}\right]&=\prod_{|u|=k_0}\mathbb{E}_{X_{u}}\left[\exp\left(-\lg\sqrt{n}W_{n-k_0}'\right)\right]\\
&\le\prod_{|u|=k_{0}}\mathbb{E}_{X_{u}}\left[\exp\left(-\lg\sqrt{n}W_{n-k_0}'\boldsymbol{1}_{\sqrt{n}W_{n-k_0}'<\varepsilon}\right)\right].
\end{align*}
In the calculations that follow, we first apply inequality \eqref{eq:deltabound} to the non-negative r.v. $W_{n-k_0}'$, then use linearity of expectation and finally the inequality $1-x\le e^{-x}$:
\begin{align*}
\mathbb{E}\left[\exp\left(-\lg \sqrt{n}W_{n,k_0}''\right)\Big{|}\FF_{k_0}\right]&\le\prod_{|u|=k_0}\mathbb{E}_{X_u}\left[1-\lg'\sqrt{n}W_{n-k_0}'\boldsymbol{1}_{\sqrt{n}W_{n-k_0}'<\varepsilon}\right]\\
&\le\prod_{|u|=k_0}\left(1-\lg'\mathbb{E}_{X_{u}}\left[\sqrt{n}W_{n-k_0}'\boldsymbol{1}_{\sqrt{n}W_{n-k_0}'<\varepsilon}\right]\right)\\
&\le\exp \left(-\sum_{|u|=k_0}\lg'\mathbb{E}_{X_{u}}\left[\sqrt{n}W_{n-k_0}'\boldsymbol{1}_{\sqrt{n}W_{n-k_0}'<\varepsilon}\right]\right).
\end{align*}
Using Fatou's lemma, we obtain :
\begin{align*}
&\limsup_{n\to\infty}\mathbb{E}\left[\exp(-\lambda\sqrt{n}W_{n,k_0}'')\Big{|}\FF_{k_0}\right]\\
&\le \exp \left(-\sum_{|u|=k_0}\lg'\liminf_{n\to\infty} \mathbb{E}_{X_{u}}\left[\sqrt{n}W_{n-k_0}'\boldsymbol{1}_{\sqrt{n}W_{n-k_0}'<\varepsilon}\right]\right)\\
&=\exp\left(\lg'\sum_{|u|=k_0}\left(-\liminf_{n\to\infty}\mathbb{E}_{X_{u}}\left[\sqrt{n}W_{n-k_0}'\right]
+\limsup_{n\to\infty}\mathbb{E}_{X_{u}}\left[\sqrt{n}W_{n-k_0}'\boldsymbol{1}_{\sqrt{n}W_{n-k_0}'\ge \varepsilon}\right]\right)\right).
\end{align*}

As seen above, by Proposition~\ref{prop:MomentKozlov}, the first term inside the summation on the right-hand side converges towards $\theta R(X_u)e^{-X_u}$ as $n\to\infty$, almost surely. Furthermore, by Proposition~\ref{prop:smallo}, there exists a positive function $h$, depending on $\eg$, such that $\frac{h(x)}{R(x)}\to0$ as $x\to\infty$ and such that for every $k_0\in\mathbb{N}$, 
\begin{equation*}
\limsup_{n\to\infty}\mathbb{E}_{X_{u}}\left[\sqrt{n}W_{n-k_0}'\boldsymbol{1}_{\sqrt{n}W_{n-k_0}'\ge \varepsilon}\right]\le h(X_u)e^{-X_u}.
\end{equation*}
Altogether this gives almost surely, for every $k_0\in\mathbb{N}$,
\begin{equation}\label{eq:upper}
\limsup_{n\to\infty}\mathbb{E}\left[\exp\left(-\lg \sqrt{n}W_{n,k_0}''\right)\Big{|}\FF_{k_0}\right]\le \exp\left(\lg'\sum_{|u|=k_0}(-\theta R(X_u)+h(X_u))e^{-X_u}\right).
\end{equation}

Since $\min_{|u|=k_0}X_{u}\to\infty$ almost surely, we get
\begin{equation}\label{eq:martren}
\lim_{k_{0}\to\infty}\sum_{|u|=k_0}\theta R(X_u)e^{-X_u}=\lim_{k_0\to\infty}\sqrt{\frac{2}{\pi\sigma^2}}D_{k_0}=\sqrt{\frac{2}{\pi\sigma^2}}D_{\infty},\ \text{a.s.},
\end{equation}
and, as a consequence, again since $\min_{|u|=k_0}X_{u}\to\infty$ almost surely,
\begin{equation*}
\lim_{k_{0}\to\infty} \sum_{|u|=k_0}h(X_u)e^{-X_u}= 0,\ \text{a.s..}
\end{equation*}
Together with \eqref{eq:upper}, this shows that
\begin{equation}
\label{eq:doublelimsup}
\limsup_{k_{0}\to\infty} \limsup_{n\to\infty}\mathbb{E}\left[\exp\left(-\lg \sqrt{n}W_{n,k_0}''\right)\Big{|}\FF_{k_0}\right] \le \exp\left(-\lg' \sqrt{\frac{2}{\pi\sigma^2}}D_{\infty}\right),\ \text{a.s.}
\end{equation}
Letting $\lg'\to\lg$ in \eqref{eq:doublelimsup} and using \eqref{eq:lower} together with \eqref{eq:martren}, we finally get for any $\lg>0$,
\begin{align*}
\lim_{k_0\to\infty}\liminf_{n\to\infty}\mathbb{E}\left[\exp\left(-\lambda \sqrt{n}W_{n,k_0}''\right)\Big{|}\FF_{k_0}\right]&=\lim_{k_0\to\infty}\limsup_{n\to\infty}\mathbb{E}\left[\exp\left(-\lambda \sqrt{n}W_{n,k_0}''\right)\Big{|}\FF_{k_0}\right]\\
&=\exp\left(-\lambda\sqrt{\frac{2}{\pi\sigma^2}}D_{\infty}\right),\ \text{a.s.}
\end{align*}

Using Cantor diagonal extraction, there exists a sequence $(k_0(n))_{n\ge0}$ (that goes to infinity as $n\to\infty$) such that for any $\lambda\in\mathbb{Q}^+=\mathbb{Q}\cap(0,\infty)$, $\mathbb{E}\left[\exp\left(-\lambda \sqrt{n}W_{n,k_0(n)}''\right)\Big{|}\FF_{k_0(n)}\right]$ converges to $\exp\left(-\lambda\sqrt{\frac{2}{\pi\sigma^2}}D_{\infty}\right)$, almost surely as $n\to\infty$. We now apply Lemma~\ref{lem:AnnexB} in Appendix~\ref{appLaplace} with $Y_{n}=\sqrt{n}W_{n,k_0(n)}''$ and $\GG_{n}=\FF_{k_0(n)}$ to get:
\begin{equation*}
\sqrt{n}W_{n,k_0(n)}''\underset{n\to\infty}{\longrightarrow}\sqrt{\frac{2}{\pi\sigma^2}}D_{\infty}\quad\text{in probability}.
\end{equation*}
Finally, we use \eqref{eq:primeprime} to see that $\sqrt{n}W_{n}$ converges to $\sqrt{\frac{2}{\pi\sigma^2}}D_{\infty}$ in probability as $n\to\infty$.
\end{proof}

\section{Proof of Proposition \ref{prop:smallo}}\label{secProofProp}

This section contains the proof of Proposition~\ref{prop:smallo}.
As is customary in this context, the main idea is to use a decomposition of the particles along the children of the spine. More precisely, let $\GG=\sigma\left(\xi_{k},X_{\xi_{k}i},k\in\mathbb{N},i\in\mathbb{N}^*\bigcup\{\varnothing\}\right)$ be the $\sigma$-algebra containing information about the spine and its children. Applying first the many-to-one formula (Proposition~\ref{prop:MTOformula}) and then Markov's inequality, we have :
\begin{align*}
\mathbb{E}_{x}\left[\sqrt{n}W_{n}'\boldsymbol{1}_{\sqrt{n}W_{n}'/\varepsilon\ge 1}\right]&=\sqrt{n}e^{-x}\mathbb{E}_{x}^{*}\left[\boldsymbol{1}_{\underset{k\le n}{\min}X_{\xi_{k}}\ge0}\boldsymbol{1}_{\sqrt{n}W_{n}'/\varepsilon\ge 1}\right],\\
&=\sqrt{n}e^{-x}\mathbb{E}_{x}^{*}\left[\boldsymbol{1}_{\underset{k\le n}{\min}X_{\xi_{k}}\ge0}\mathbb{E}_x^*\left[\boldsymbol{1}_{\sqrt{n}W_{n}'/\varepsilon\ge 1}\Big{|}\GG\right]\right]\\
&\le\sqrt{n}e^{-x}\mathbb{E}_{x}^{*}\left[\boldsymbol{1}_{\underset{k\le n}{\min}X_{\xi_{k}}\ge0}\mathbb{E}_{x}^*\left[\frac{\sqrt{n}W_{n}'}{\varepsilon}\wedge 1\Big{|}\GG\right]\right],\\
&=\sqrt{n}e^{-x}\mathbb{P}^{*}_{x}\left(\underset{k\le n}{\min}X_{\xi_{k}}\ge0\right)\mathbb{E}_{x}^{*}\left[\mathbb{E}_{x}^*\left[\frac{\sqrt{n}W_{n}'}{\varepsilon}\wedge 1\Big{|}\GG\right]\Bigg{|}\underset{k\le n}{\min}X_{\xi_{k}}\ge0\right].
\end{align*}
Using Lemma \ref{lem:Kozlov}, we obtain
\begin{equation*}
\mathbb{E}_{x}\left[\sqrt{n}W_{n}'\boldsymbol{1}_{\frac{\sqrt{n}W_{n}'}{\varepsilon}\ge 1}\right]\le \theta' R(x)e^{-x}\mathbb{E}_{x}^{*}\left[\mathbb{E}_{x}^*\left[\frac{\sqrt{n}W_{n}'}{\varepsilon}\wedge 1\Big{|}\GG\right]\Bigg{|}\underset{k\le n}{\min}X_{\xi_{k}}\ge0\right],
\end{equation*}
where $\theta'$ is the constant from Lemma~\ref{lem:Kozlov}.

\begin{lemma}[Decomposition of $W_{n}'$ along the spine]\label{lem:DecompositionAlongSpine}
~\\
We have, $\mathbb{P}_{x}^{*}$ a.s., that
\begin{align*}
\mathbb{E}_{x}^*\left[\frac{\sqrt{n}W_{n}'}{\varepsilon}\wedge 1\Big{|}\GG\right]\le&\left(\frac{\sqrt{n}}{\varepsilon}e^{-X_{\xi_{n}}}\right)\wedge 1\\
&+\left(\frac{\sqrt{n}}{\varepsilon}\sum_{k=0}^{n-1}\sum_{\substack{i\in\mathbb{N} \\ \xi_{k}i\ne\xi_{k+1}}}\mathbb{E}_{X_{\xi_{k}i}}[W_{n-k-1}']\right)\wedge 1
\end{align*}
\end{lemma}
\begin{proof}[Proof of Lemma \ref{lem:DecompositionAlongSpine}]
Remember the definition of $W_{n}'$ in Equation~\eqref{eq:WnPrime}. We decompose this expression using the spine and its children, along with the branching property in order to get the following identity :
\begin{equation}
W_{n}'=\boldsymbol{1}_{\underset{k\le n}{\min}X_{\xi_k}\ge0}e^{-X_{\xi_n}}+\sum_{k=0}^{n-1}\sum_{\substack{i\in\mathbb{N}\\\xi_{k}i\ne\xi_{k+1}}}\boldsymbol{1}_{\forall v\le\xi_{k}i,X_{v}\ge0}W_{n-k-1}'(k,i),
\end{equation}
where $W_{n-k-1}'(k,i)$ has the law of $W_{n-k-1}'$ under $\mathbb{P}_{X_{\xi_k i}}$ conditioned on $\GG$. We will use the fact that for any random variables $X$ and $X'$, by subadditivity of $x\wedge 1$ and Jensen's inequality for concave functions,
\begin{equation}\label{eq:wedge}
\mathbb{E}\left[\left(X+X'\right)\wedge 1\right]\le\mathbb{E}\left[X\wedge 1+X'\wedge 1\right]\le\mathbb{E}[X]\wedge 1+\mathbb{E}[X']\wedge1.
\end{equation}
Now we condition on $\GG$, and using inequality \eqref{eq:wedge}, we get :
\begin{align*}
\mathbb{E}_{x}^*\left[\frac{\sqrt{n}W_{n}'}{\varepsilon}\wedge 1\Big{|}\GG\right]\le&\ \mathbb{E}_{x}^*\left[\frac{\sqrt{n}}{\varepsilon}e^{-X_{\xi_n}}\boldsymbol{1}_{\underset{k\le n}{\min}X_{\xi_k}\ge0}\Big{|}\GG\right]\wedge 1\\
&+\mathbb{E}_{x}^*\left[\frac{\sqrt{n}}{\varepsilon}\sum_{k=0}^{n-1}\sum_{\substack{i\in\mathbb{N}\\\xi_{k}i\ne\xi_{k+1}}}\boldsymbol{1}_{\forall v\le\xi_{k}i,X_{v}\ge0}W_{n-k-1}'(k,i)\Big{|}\GG\right]\wedge 1\\
\le&\ \left(\frac{\sqrt{n}}{\varepsilon}e^{-X_{\xi_n}}\boldsymbol{1}_{\underset{k\le n}{\min}X_{\xi_k}\ge0}\right)\wedge 1\\
&+\left(\frac{\sqrt{n}}{\varepsilon}\sum_{k=0}^{n-1}\sum_{\substack{i\in\mathbb{N}\\\xi_{k}i\ne\xi_{k+1}}}\boldsymbol{1}_{\forall v\le\xi_{k}i,X_{v}\ge0}\mathbb{E}_{X_{\xi_{k}i}}\left[W_{n-k-1}'\right]\right)\wedge 1.
\end{align*}
We end the proof by bounding the indicator functions by $1$.
\end{proof}

Applying Lemma \ref{lem:DecompositionAlongSpine}, we obtain the following bound :
\begin{equation}\label{eq:T1plusT2}
\mathbb{E}_{x}\left[\sqrt{n}W_{n}'\boldsymbol{1}_{\frac{\sqrt{n}W_{n}'}{\varepsilon}\ge 1}\right]\le \theta' R(x)e^{-x}\left(T_{1}(x,\varepsilon,n)+T_{2}(x,\varepsilon,n)\right),
	\end{equation}
where
\begin{align*}
T_{1}(x,\varepsilon,n)&=\mathbb{E}_{x}^{*}\left[\left(\frac{\sqrt{n}}{\varepsilon}e^{-X_{\xi_{n}}}\right)\wedge 1\Big{|}\underset{k\le n}{\min}X_{\xi_k}\ge0\right],
\\
T_{2}(x,\varepsilon,n)&=\mathbb{E}_{x}^{*}\left[\left(\frac{\sqrt{n}}{\varepsilon}\sum_{k=0}^{n-1}\sum_{\substack{i\in\mathbb{N}\\\xi_{k}i\ne\xi_{k+1}}}\mathbb{E}_{X_{\xi_{k}i}}\left[W_{n-k-1}'\right]\right)\wedge 1\Big{|}\underset{k\le n}{\min}X_{\xi_k}\ge 0\right].
\end{align*}

We state two lemmas to control those terms:

\begin{lemma}
\label{lem:SmallTerm}
For any fixed $\varepsilon>0$ and $x\ge0$,
\begin{equation}
T_{1}(x,\varepsilon,n)\underset{n\to\infty}{\longrightarrow}0.
\end{equation}
\end{lemma}
\begin{proof}[Proof of Lemma \ref{lem:SmallTerm}]
~\\
By Iglehart \cite{Iglehart1974} and Bolthausen \cite{Bolthausen1976}, we know that $\frac{X_{\xi_{n}}}{\sqrt{n}}$ converges in distribution to a positive random variable under the conditioned probability $\mathbb{P}_{x}^*\left(\cdotp\Big{|}\underset{k\le n}{\min}X_{\xi_k}\ge0\right)$, so $\sqrt{n}e^{-X_{\xi_n}}$ converges in distribution to 0 under the same conditioning. Moreover, the random variables $\left(\frac{\sqrt{n}}{\varepsilon}e^{-X_{\xi_n}}\right)\wedge 1$ are trivially bounded by 1. Hence
\begin{equation}
\mathbb{E}_{x}^*\left[\left(\frac{\sqrt{n}}{\varepsilon}e^{-X_{\xi_n}}\right)\wedge 1\Big{|}\underset{k\le n}{\min}X_{\xi_k}\ge0\right]\underset{n\to\infty}{\longrightarrow}0,
\end{equation}
which gives us that $T_{1}(x,\varepsilon,n)\underset{n\to\infty}{\longrightarrow}0$.
\end{proof}

\begin{lemma}
\label{lem:BigTerm}
For every $\eg>0$, there exists a positive function $\tilde h$,  such that $\tilde h(x)\to 0$ as $x\to\infty$ and such that the following holds: for every $x\ge0$, we have
\begin{equation}
\limsup_{n\to\infty} T_{2}(x,\varepsilon,n)\le \tilde h(x).
\end{equation}
\end{lemma}
The proof of Lemma~\ref{lem:BigTerm} is delayed to the next section.

We can now finish the proof of Proposition~\ref{prop:smallo}.
\begin{proof}[Proof of Proposition~\ref{prop:smallo}]
~\\
Applying Lemmas \ref{lem:SmallTerm} and \ref{lem:BigTerm} to Equation \eqref{eq:T1plusT2}, we get that for every $\eg>0$ and $x\ge0$,
\begin{equation*}
\limsup_{n\to\infty} \mathbb{E}_{x}\left[\sqrt{n}W_{n}'\boldsymbol{1}_{\frac{\sqrt{n}W_{n}'}{\varepsilon}\ge 1}\right]\le\theta' R(x)\tilde h(x)e^{-x},
\end{equation*}
which implies the proposition.
\end{proof}

\section{Proof of Lemma \ref{lem:BigTerm}}
\label{secGrosTerme}

In order to prove Lemma \ref{lem:BigTerm}, we introduce in Lemma~\ref{lem:KV} a result on convergence of functionals of the branching random walk conditioned on the spine staying above the origin until time $n$ to corresponding functionals of the process conditioned on the spine staying above the origin for all time. It is inspired by a analogous result for random walks by Kersting and Vatutin \cite[Lemma~5.2]{Kersting2017}.
We recall that $R$ is harmonic for the sub-Markov process obtained by killing $(X_{\xi_{n}})_{n\ge0}$ when entering $(-\infty,0)$ (Proposition~\ref{prop:Rharmonic}). 
For $x\ge0$, define the probability measure $\mathbb{P}^{+}_{x}$ by
\begin{equation}\label{eq:defP+}
\frac{d\mathbb{P}_x^+}{d\mathbb{P}_x^*}\Big|_{\mathcal \FF_n} = \frac{1}{R(x)} R(X_{\xi_n}) \boldsymbol 1_{\min_{k\le n}X_{\xi_{k}}\ge0},
\end{equation}
and denote the associated expectation by $\mathbb{E}_{x}^{+}$. Heuristically, under $\mathbb{P}_{x}^+$, the motion of the spine is conditioned to stay non-negative at all times.

\begin{lemma}
\label{lem:KV}
With the above definitions, let $(Y_n)_{n\ge0}$ be a uniformly bounded sequence of random variables, adapted to $(\FF_n)_{n\ge0}$. Let $x\ge0$.
If $Y_\infty$ is a random variable such that $Y_n\underset{n\to\infty}{\longrightarrow}Y_{\infty}$ in probability under $\mathbb{P}_{x}^{+}$, then 
\begin{equation}
\lim_{n\to\infty}\mathbb{E}_x^{*}\left[Y_{n}\big{|}\min_{k\le n}X_{\xi_{k}}\ge 0\right]=\mathbb{E}_x^{+}\left[Y_{\infty}\right].
\end{equation}
\end{lemma}

\begin{proof}[Proof of Lemma~\ref{lem:KV}]
~\\
We give the proof for completeness, which follows almost exactly along the lines of Lemma~5.2 in Kersting and Vatutin \cite{Kersting2017}. Throughout the proof, fix $x\ge0$. 

Define 
\begin{equation*}
m_n(x)=\mathbb{P}_{x}^{*}\left(\min_{0\le k\le n}X_{\xi_k}\ge0\right).
\end{equation*} 
Then, for every $l\le n$, conditioning on $\mathcal{F}_{l}$ gives 
\begin{align*}
\mathbb{E}_{x}^{*}\left[Y_{l}\Big{|}\min_{k\le n}X_{\xi_k}\ge0\right]&=\mathbb{E}_{x}^{*}\left[Y_{l}\frac{1}{m_{n}(x)}\boldsymbol 1_{\min_{k\le n}X_{\xi_{k}}\ge0}\right]\\
&=\mathbb{E}_{x}^{*}\left[Y_{l}\frac{m_{n-l}(X_{\xi_l})}{m_{n}(x)}\boldsymbol 1_{\min_{k\le l}X_{\xi_{k}}\ge0}\right].
\end{align*}
Now fix $l\in\mathbb{N}$ for the moment. By Lemma~\ref{lem:Kozlov}, the ratio $m_{n-l}(X_{\xi_l})/m_{n}(x)$ converges to $R(X_{\xi_l})/R(x)$ as $n\to\infty$ and is bounded by a constant multiple of this quantity. By dominated convergence, we get
\begin{align}
\lim_{n\to\infty}\mathbb{E}_{x}^{*}\left[Y_{l}\Big{|}\min_{k\le n}X_{\xi_k}\ge0\right]&=\mathbb{E}_{x}^{*}\left[Y_{l}\lim_{n\to\infty}\frac{m_{n-l}(X_{\xi_l})}{m_{n}(x)}\boldsymbol 1_{\min_{k\le l}X_{\xi_{k}}\ge0}\right]\nonumber\\
&=\mathbb{E}_{x}^{*}\left[Y_{l}\frac{R(X_{\xi_l})}{R(x)}\boldsymbol 1_{\min_{k\le l}X_{\xi_{k}}\ge0}\right]\nonumber\\
&=\mathbb{E}_{x}^{+}[Y_l],\label{eq:proofLemKV}
\end{align}
where the last equality follows from the definition of $\mathbb{P}_{x}^{+}$ in Equation~\eqref{eq:defP+}.

Now let $a>1$. Using again Lemma~\ref{lem:Kozlov}, we get that for some constant $K$, for every $1\le l\le n$,
\begin{align}
\Bigg{|}\mathbb{E}_{x}^{*}\left[Y_{n}-Y_{l}\Big{|}\min_{k\le \lfloor an\rfloor}X_{\xi_k}\ge0\right]\Bigg{|}&\le\mathbb{E}_{x}^{*}\left[|Y_n-Y_l|\frac{m_{\lfloor (a-1)n\rfloor}(X_{\xi_n})}{m_{\lfloor an\rfloor}(x)}\boldsymbol 1_{\min_{k\le n}X_{\xi_{k}}\ge0}\right]\nonumber\\
&\le K\sqrt{\frac{a}{a-1}}\mathbb{E}_{x}^{*}\left[|Y_n-Y_l|\frac{R(X_{\xi_n})}{R(x)}\boldsymbol 1_{\min_{k\le n}X_{\xi_{k}}\ge0}\right]\nonumber\\
&= K\sqrt{\frac{a}{a-1}}\mathbb{E}_{x}^{+}\left[|Y_n-Y_l|\right].\label{eq:proofLemKV2}
\end{align}
Now recall that by assumption, $(Y_n)_{n\ge0}$ is uniformly bounded and converges to a limit $Y_\infty$ in probability, under $\mathbb{P}_x^+$. Hence, the following holds:
\begin{itemize}
\item $\mathbb{E}_x^+[Y_l] \to \mathbb{E}_x^+[Y_\infty]$ as $l\to\infty$, and
\item $\mathbb{E}_{x}^{+}[|Y_n-Y_l|] \to 0$, as $n$ then $l$ tends to infinity.
\end{itemize}
Combining these two facts with \eqref{eq:proofLemKV} and  \eqref{eq:proofLemKV2} and letting first $n$ then $l$ go to infinity, we obtain
\begin{equation}
\label{eq:kv1}
\lim_{n\to\infty} \mathbb{E}_{x}^{*}\left[Y_{n}\Big{|}\min_{k\le \lfloor an\rfloor}X_{\xi_k}\ge0\right] = \mathbb{E}_{x}^{+}[Y_\infty].
\end{equation}

Now set $M = \sup_n \|Y_n\|_\infty$ and note that $|Y_\infty| \le M$ almost surely. Then,
\begin{align*}
&\left|\mathbb{E}_{x}^{*}\left[Y_{n}\boldsymbol 1_{\min_{k\le n}X_{\xi_k}\ge0}\right]-\mathbb{E}_{x}^{+}[Y_\infty]m_n(x)\right|\\
&\le\left|\mathbb{E}_{x}^{*}\left[Y_n \boldsymbol 1_{\min_{k\le \lfloor an\rfloor}X_{\xi_k}\ge0}\right]-\mathbb{E}_{x}^{+}[Y_\infty]m_{\lfloor an\rfloor}(x)\right|+2M\left(m_n(x) - m_{\lfloor an\rfloor}(x)\right).
\end{align*}
and dividing by $m_n(x)$, we get
\begin{align*}
&\left|\mathbb{E}_{x}^{*}\left[Y_{n}\Big|\min_{k\le n}X_{\xi_k}\ge0\right]-\mathbb{E}_{x}^{+}[Y_\infty]\right|\\
&\le\frac{m_{\lfloor an\rfloor}(x)}{m_n(x)}\left|\mathbb{E}_{x}^{*}\left[Y_n\Big|\min_{k\le \lfloor an\rfloor}X_{\xi_k}\ge0\right]-\mathbb{E}_{x}^{+}[Y_\infty]\right|+2M\left(1 - \frac{m_{\lfloor an\rfloor}(x)}{m_n(x)}\right).
\end{align*}
Note that $m_{\lfloor an\rfloor}(x)/m_n(x) \to 1/\sqrt a$ as $n\to\infty$ by Lemma~\ref{lem:Kozlov}. Hence, using \eqref{eq:kv1}, we get
\[
\limsup_{n\to\infty} \left|\mathbb{E}_{x}^{*}\left[Y_{n}\Big|\min_{k\le n}X_{\xi_k}\ge0\right]-\mathbb{E}_{x}^{+}[Y_\infty]\right| \le 2M\left(1 - \frac 1 {\sqrt{a}}\right).
\]
Letting $a\to 1$ gives
\[
\lim_{n\to\infty} \mathbb{E}_{x}^{*}\left[Y_{n}\Big|\min_{k\le n}X_{\xi_k}\ge0\right] = \mathbb{E}_{x}^{+}[Y_\infty],
\]
which was to be proven.
\end{proof}

We will furthermore need the following estimate on the potential kernel of the spine under the law $\mathbb{P}_x^+$:

\begin{lemma}\label{lem:GreenEstimate}
 Let $f:\mathbb{R}_+\to\mathbb{R}_+$ be a bounded, non-increasing function satisfying $\int_0^\infty yf(y)\,dy < \infty$. Then
 \[
  \mathbb{E}_x^+\left[\sum_{k=0}^\infty f(X_{\xi_k}) \right]\to 0,\quad \text{as $x\to\infty$}.
 \]
 Furthermore, the above expectation is finite for every $x\ge0$.
\end{lemma}

\begin{proof}[Proof of Lemma \ref{lem:GreenEstimate}]
~\\
By the definition of the law $\mathbb{P}_x^+$, we have
\begin{align}
\mathbb{E}_{x}^{+}\left[\sum_{k=0}^{\infty}f(X_{\xi_k})\right] = \frac{1}{R(x)}\mathbb{E}_{x}^{*}\left[\sum_{k=0}^{\infty}R(X_{\xi_k})f(X_{\xi_k})\boldsymbol 1_{\forall l\le k,X_{\xi_l}\ge 0}\right].\label{eq:sumf}
\end{align}
Let $\mu$ and $\hat \mu$ be the renewal measures associated to the (absolute values of the) strictly descending and strictly ascending ladder heights of the random walk $(X_{\xi_n})_{n\ge0}$, respectively, see Appendix~\ref{appGreen}. Since the random walk has finite variance by \eqref{eq:MomentsSpine}, the ladder heights have finite expectation, see e.g. Rogozin~\cite{Rogozin1964}. Recall that $R(x) = \mu([0,x])$ and define $\hat R(x) = \hat \mu([0,x])$. Also recall the constant $c^=$ from Appendix~\ref{appGreen}. Using Theorem~\ref{th:green} we get
\begin{align}
\mathbb{E}_{x}^{*}\left[\sum_{k=0}^{\infty}R(X_{\xi_k})f(X_{\xi_k})\boldsymbol 1_{\forall l\le k,X_{\xi_l}\ge 0}\right]&=c^=\int_{z=0}^{\infty}\int_{y=0}^{x}R(x-y+z)f(x-y+z)\mu(dy)\hat{\mu}(dz)\nonumber\\
&=:I(x).\label{eq:I0}
\end{align}

Define the function $\tilde{f}:y\mapsto (1+y)f(y)$, $y\ge0$. By \eqref{eq:defc1}, there exists $C\in(0,\infty)$, such that
\begin{equation}
\label{eq:I}
I(x)\le C\int_{z=0}^{\infty}\int_{y=0}^{x}\tilde{f}(x-y+z)\mu(dy)\hat{\mu}(dz).
\end{equation}
In order to bound this double integral, we first integrate over $z$, then over $y$. For $w\ge 0$, put $g(w)=\int_{0}^{\infty}\tilde{f}(w+z)\hat{\mu}(dz)$. Then \eqref{eq:I} implies
\begin{equation}
 \label{eq:I2}
 I(x)\le C\int_{0}^{x}g(x-y)\mu(dy).
\end{equation}

We first bound $g$. For simplicity, suppose $\hat\mu$ is non-arithmetic or that its span is equal to 1, the general case can be treated by a scaling argument. We then have for every $w\ge 0$, since $f$ is non-increasing,
\begin{align*}
g(w)&=\sum_{k=0}^{\infty}\int_{k}^{k+1}\tilde{f}(w+z)\hat{\mu}(dz)\\
&\le \sum_{k=0}^{\infty}(w+k+2)f(w+k)\left(\hat{R}(k+1)-\hat{R}(k)\right).
\end{align*}
By Feller's renewal theorem, $\hat{R}(k+1)-\hat{R}(k)$ is uniformly bounded in $k$, hence,
\begin{align*}
 g(w) &\le C\sum_{k=0}^{\infty}(w+k+2)f(w+k)\\
 &\le \int_{w-1}^\infty (z+3)f(z\vee 0)\,dz,
\end{align*}
using again that $f$ is non-increasing. The hypotheses on $f$ now readily imply that $g$ is bounded and
\begin{align}
\label{eq:gzero}
 g(w) \to 0,\quad w\to\infty.
\end{align} 

By \eqref{eq:sumf}, \eqref{eq:I0} and \eqref{eq:I2}, it remains to show that
\begin{equation}
 \label{eq:g}
 \frac 1 {R(x)} \int_{0}^{x}g(x-y)\mu(dy) \to 0,\quad x\to\infty.
\end{equation}
Let $\dg>0$. By \eqref{eq:gzero}, there exists $y_0\ge 0$ such that $\forall y\ge y_0, g(y)\le \dg$. Then,
\begin{align*}
\int_{0}^{x}g(x-y)\mu(dy) &\le \dg \mu([0,x])+\int_{x-y_0}^{x}g(x-y)\mu(dy).
\end{align*}
The first term on the right-hand side equals $\delta R(x)$ by definition and the second term converges to a constant as $x\to\infty$, by the key renewal theorem (see Feller \cite[p.~363]{Feller1971}). As a consequence,
\[
 \limsup_{x\to\infty} \frac 1 {R(x)} \int_{0}^{x}g(x-y)\mu(dy) \le \delta.
\]
Since $\delta$ was arbitrary, this proves \eqref{eq:g} and thus finishes the proof.
\end{proof}

We now have all the tools we need to prove the main lemma of this section.

\begin{proof}[Proof of Lemma \ref{lem:BigTerm}]
~\\
Throughout the proof, we fix $\eg>0$.

As mentioned above, our goal is to apply Lemma \ref{lem:KV} to a suitable sequence of random variables $(Y_{n})_{n\ge0}$. Recall that
\begin{equation*}
T_{2}(x,\varepsilon,n)=\mathbb{E}_{x}^{*}\left[\left(\frac{\sqrt{n}}{\varepsilon}\sum_{k=0}^{n-1}\sum_{\substack{i\in\mathbb{N}\\\xi_{k}i\ne\xi_{k+1}}}\mathbb{E}_{X_{\xi_{k}i}}\left[W_{n-k-1}'\right]\right)\wedge 1\Big{|}\underset{k\le n}{\min}X_{\xi_k}\ge 0\right].
\end{equation*}
Using Proposition \ref{prop:MomentKozlov}, we get the following bound:
\begin{equation}\label{eq:boundfracsqrt}
\frac{\sqrt{n}}{\varepsilon}\sum_{k=0}^{n-1}\sum_{\substack{i\in\mathbb{N}\\\xi_{k}i\ne\xi_{k+1}}}\mathbb{E}_{X_{\xi_{k}i}}\left[W_{n-k-1}'\right]\le\frac{\theta'}{\varepsilon}\sum_{k=0}^{n-1}\sqrt{\frac{n}{n-k}}\sum_{\substack{i\in\mathbb{N}\\\xi_{k}i\ne\xi_{k+1}}}R(X_{\xi_{k}i})e^{-X_{\xi_{k}i}},
\end{equation}
with $\theta'$ the constant from Proposition~\ref{prop:MomentKozlov}.

In order to bound the contribution of the children of the spine to the right-hand side of Equation~\eqref{eq:boundfracsqrt}, define a sequence of random variables for $k\ge0$:
\begin{equation}\label{eq:defVk}
V_{k}=\sum_{\substack{i\in\mathbb{N}\\\xi_{k}i\ne\xi_{k+1}}}\left(1+(X_{\xi_{k}i}-X_{\xi_{k}})^+\right)e^{-(X_{\xi_{k}i}-X_{\xi_k})}.
\end{equation}

Using Lemma \ref{lem:BoundR}, we have for every $k\ge0$:
\begin{align*}
\sum_{\substack{i\in\mathbb{N}\\\xi_{k}i\ne\xi_{k+1}}}R(X_{\xi_{k}i})e^{-X_{\xi_{k}i}}
&=\sum_{\substack{i\in\mathbb{N}\\\xi_{k}i\ne\xi_{k+1}}}R\left(X_{\xi_{k}}+(X_{\xi_{k}i}-X_{\xi_{k}})\right)e^{-X_{\xi_{k}}-(X_{\xi_{k}i}-X_{\xi_k})}\\
&\le c_1 \left(1+X_{\xi_{k}}^{+}\right)e^{-X_{\xi_k}}\sum_{\substack{i\in\mathbb{N}\\\xi_{k}i\ne\xi_{k+1}}}\left(1+(X_{\xi_{k}i}-X_{\xi_{k}})^+\right)e^{-(X_{\xi_{k}i}-X_{\xi_k})},\\
&\le c_1 \left(1+X_{\xi_{k}}^{+}\right)e^{-X_{\xi_k}}V_{k},\\
&\le C e^{-X_{\xi_k}/2}V_k,
\end{align*}
where $C>0$ is such that $c_{1}(1+x^+)e^{-x}\le C e^{-x/2}$ for all $x\ge0$.

Put $Y_n=\left(\sum_{k=0}^{n-1}\left(\sqrt{\frac{n}{n-k}}\frac{e^{-X_{\xi_{k}}/2}V_{k}}{\varepsilon}\wedge 1\right)\right)\wedge 1$.
Using the previous inequalities and equivalents, and plugging it in the expression of $T_{2}(x,\varepsilon,n)$, we obtain :
\begin{equation}\label{eq:boundingT2}
T_{2}(x,\varepsilon,n)\le \theta' C \mathbb{E}_{x}^{*}\left[Y_{n}\Big{|}\underset{k\le n}{\min}X_{\xi_{k}}\ge0\right].
\end{equation}
In order to bound the expectation on the right-hand side of \eqref{eq:boundingT2}, we first bound $Y_n$ by
\begin{equation}
 \label{eq:YnDecomposition}
 Y_n \le \sqrt 2 Y_n' + Y_n'',
\end{equation}
where
\[
 Y_n' = \left( \sum_{k=0}^{\lfloor n/2\rfloor}\left( \frac{e^{-X_{\xi_{k}}/4}V_{k}}{\varepsilon}\wedge 1\right)\right)\wedge 1,\quad Y_n'' = \left( \sum_{k=\lfloor n/2\rfloor+1}^{n-1}\left(\sqrt n \frac{e^{-X_{\xi_{k}}/2}V_{k}}{\varepsilon}\wedge 1\right)\right)\wedge 1,
\]
(the $4$ in the exponent in the definition of $Y_n'$ is unimportant and serves to make notation simpler later on).
Note that both sequences of random variables $(Y_n')_{n\ge0}$ and $(Y_n'')_{n\ge0}$ are adapted to the canonical filtration $(\FF_n)_{n\ge0}$ of the branching random walk and uniformly bounded by $1$. By monotonicity, $Y_n'$ converges $\mathbb{P}_{x}^+$-almost surely as $n\to\infty$ to $Y'_{\infty}$ defined by 
\begin{equation}\label{eq:defYinfty}
Y'_{\infty} = \left(\sum_{k=0}^{\infty}\left(\frac{e^{-X_{\xi_{k}}/4}V_{k}}{\varepsilon}\wedge 1\right)\right)\wedge 1.
\end{equation}
We now claim the following:
\begin{enumerate}
 \item[a)] $\mathbb{E}_x^+[Y_\infty'] \to 0$ as $x\to\infty$.
 \item[b)] $Y_n'' \to 0$ as $n\to\infty$, in $\mathbb{P}_{x}^+$-probability.
\end{enumerate}
Let us see how these two claims imply the statement of the lemma. First, applying Lemma~\ref{lem:KV} to the r.v.'s $(Y_n')_{n\ge0}$ and $Y_\infty'$, we have
\begin{equation*}
\lim_{n\to\infty}\mathbb{E}_{x}^{*}\left[Y_n'\big{|}\min_{k\le n}X_{\xi_k}\ge 0\right]=\mathbb{E}_{x}^{+}\left[Y_{\infty}'\right].
\end{equation*}
Second, applying again Lemma~\ref{lem:KV} the r.v.'s $(Y_n'')_{n\ge0}$ and using claim b) above, we get
\[
 \lim_{n\to\infty}\mathbb{E}_{x}^{*}\left[Y_n''\big{|}\min_{k\le n}X_{\xi_k}\ge 0\right]= 0.
\]
Plugging these two equalities into \eqref{eq:YnDecomposition} and \eqref{eq:boundingT2} yields (with some $C>0$),
\[
 \limsup_{n\to\infty} T_{2}(x,\varepsilon,n)\le C \mathbb{E}_x^+[Y_\infty'].
\]
Together with Claim a) this yields the lemma.

It now remains to prove Claims a) and b) above. We start with Claim a). Recall that under $\mathbb{P}_{x}^+$, the offspring distribution is biased by
\begin{align*}
\frac{\sum_{|u|=1}R(X_u)e^{-X_u}}{R(x)e^{-x}}&\le C\frac{\sum_{|u|=1}(R(x)+(X_u-x)^+)e^{-(X_u-x)}}{R(x)}\\
&=C\left(W_1+\frac{Z_1}{R(x)}\right),
\end{align*}
where the inequality is an easy consequence of \eqref{eq:Rlimit} in Lemma~\ref{lem:Kozlov}.
Hence,  for every $k\ge 0$,
\begin{multline*}
 \qquad\mathbb{E}_{x}^+\left[\frac{e^{-X_{\xi_k}/4}V_k}{\eg}\wedge 1\big{|}\FF_{k}\right]\le C f(X_{\xi_k}),\quad\text{where}\\
 f(y) = \mathbb{E}\left[\left(W_1+\frac{Z_1}{R(y)}\right)\left(\frac{e^{-y/4}(W_1+Z_1)}{\eg}\wedge 1\right)\right].\qquad
\end{multline*}
We decompose $f$:
\begin{align*}
 f(y) &= f_{1}(y)+\frac{f_{2}(y)}{R(y)},\quad\text{with}\\
 f_1(y) &=  \mathbb{E}\left[W_1\left(\eg^{-1} e^{-y/4}(W_1+Z_1)\wedge 1\right)\right],\\
 f_2(y) &= \mathbb{E}\left[Z_1\left(\eg^{-1} e^{-y/4}(W_1+Z_1)\wedge 1\right)\right].
\end{align*}
We now use Assumptions \eqref{eq:Wlog2W} and \eqref{eq:DlogD} and  a Tauberian-type result (Lemma~\ref{lem:tauberian} in the appendix) to bound certain integrals of $f_1$ and $f_2$. First note that it can be obtained from Lemma B.1 (i) in Aïdékon~\cite{Aidekon2013} (and is implicit in the proof of part (ii) of that lemma) that Assumptions \eqref{eq:Wlog2W} and \eqref{eq:DlogD}  imply that
\begin{equation*}
\mathbb{E}\left[W_1\left(\log_+(W_1+Z_1)\right)^2\right]<\infty\quad\text{and}\quad\mathbb{E}\left[Z_1\log_+(W_1+Z_1)\right]<\infty.
\end{equation*}
We then apply Lemma~\ref{lem:tauberian} twice to the r.v. $W_1+Z_1$, once under the law $\mathbb{E}[W_1\,\cdot]$ and with $\rho(x) = x$, and once under the law $\mathbb{E}[(Z_1/\mathbb{E}[Z_1])\,\cdot]$ and with $\rho \equiv 1$. We then obtain 
\begin{equation*}
\int_{0}^{\infty}f_{1}(y)y\,dy<\infty\quad\text{and}\quad\int_{0}^{\infty}f_{2}(y)\,dy<\infty,
\end{equation*}
hence, by the bound \eqref{eq:defc1} on $R$, \begin{equation}\label{eq:fintegral}\int_{0}^{\infty}f(y)y\,dy<\infty.\end{equation}

Now we may compute:
\begin{align}
\mathbb{E}_{x}^{+}\left[Y_\infty\right]
\le \sum_{k=0}^{\infty} \mathbb{E}_{x}^{+}\left[\frac{e^{-X_{\xi_k}/4}V_k}{\eg}\wedge 1\right]
\le C \sum_{k=0}^{\infty}\mathbb{E}_{x}^{+}\left[f(X_{\xi_k})\right].\label{eq:boundingYinfty}
\end{align}
Note that $f$ is bounded and non-increasing by definition. Equations~\eqref{eq:fintegral} and \eqref{eq:boundingYinfty} together with Lemma~\ref{lem:GreenEstimate} then imply Claim a).

To prove Claim b) we use the following invariance principle by Caravenna and Chaumont \cite{Caravenna2008}: the rescaled process $(n^{-1/2}X_{\xi_{\lfloor nt\rfloor}})_{t\ge0}$ converges in distribution under $\mathbb{P}_{x}^+$ to a three-dimensional Bessel process as $n\to\infty$. As a consequence, for every $\eta\in(0,1)$ there exists $\dg>0$, such that for large $n$, with probability at least $1-\eta$, we have $X_{\xi_k}>\delta\sqrt{n}$ for every $k\in[\![\lfloor n/2\rfloor+1,n]\!]$. So there is some positive constant $c$ such that, with probability at least $1-\eta$,
\begin{equation*}
Y_n''\le c\sum_{k=\lfloor n/2\rfloor+1}^{n-1}\left(\frac{e^{-X_{\xi_{k}}/4}V_{k}}{\varepsilon}\wedge 1\right),
\end{equation*}
 which converges to $0$ in $\mathbb{P}_x^+$-probability as $n\to\infty$ since for every $x\ge0$, as shown above,
\begin{equation*}
\mathbb{E}_{x}^{+}\left[\sum_{k=0}^{\infty}\left(\frac{e^{-X_{\xi_k}/4}V_k}{\eg}\wedge1\right)\right]<\infty.
\end{equation*}
This proves Claim b) and finishes the proof of the Lemma.

\end{proof}


\appendix

\section{Random walks on the half-line}\label{appGreen}
In this section, we recall some properties of random walks on the positive half-line, including a representation of their Green operators (Theorem~\ref{th:green}). This allows to prove Lemma~\ref{lem:GreenEstimate} from the main text.

Let $(S_n)_{n\ge0}$ be a random walk of oscillating type started at $S_0 = 0$. Denote its Markov dual by $\hat S_n = -S_n$. Associated to the random walk are four ladder height processes with associated ladder times:
\begin{itemize}
\item Strictly descending: $(H_n)_{n\ge0}$, $(L_n)_{n\ge0}$
\item Weakly descending: $(H_n^=)_{n\ge0}$, $(L_n^=)_{n\ge0}$
\item Strictly ascending: $(\hat H_n)_{n\ge0}$, $(\hat L_n)_{n\ge0}$
\item Weakly ascending: $(\hat H_n^=)_{n\ge0}$, $(\hat L_n^=)_{n\ge0}$
\end{itemize}
Explicitly, $H_n = S_{L_n}$, where $L_0 = 0$ and 
\[
L_{n+1} = \min\{k>L_n: S_k < S_{L_n}\}.
\]
The other processes are defined analogously, with the ``$<$'' above replaced by, respectively, $\le$, $>$, $\ge$.

Let $\mu$, $\mu^=$, $\hat \mu$, $\hat \mu^=$ be the renewal measures of the processes $(|H_n|)_{n\ge0}$, $(|H_n^=|)_{n\ge0}$, $(\hat H_n)_{n\ge0}$, $(\hat H_n^=)_{n\ge0}$, respectively, i.e.,
\[
\mu(dx) = \sum_{n\ge0} \mathbb{P}(|H_n|\in dx),
\]
with the other measures defined analogously. We first note the following fact:
\begin{lemma}
\label{lem:c=}
There exists a constant $c^=\in(0,\infty)$, such that
\[
\mu^= = c^=\mu,\quad \hat\mu^= = c^=\hat \mu.
\]
Furthermore, $c^=$ admits the following equivalent expressions
\begin{align*}
c^=
&= \left(1-\sum_{n=1}^\infty \mathbb{P}(S_n = 0,\,S_k < 0,\,1\le k \le n)\right)^{-1}\\
&= \left(1-\sum_{n=1}^\infty  \mathbb{P}(S_n = 0,\,S_k > 0,\,1\le k \le n)\right)^{-1}\\
&= \exp\left(\sum_{n=1}^\infty \frac 1 n \mathbb{P}(S_n = 0)\right)
\end{align*}
\end{lemma}
\begin{proof}
The first statement (with $c^=$ equal to its first, respectively, second  expression given in the statement) is (1.13) in \cite[Section XII.1]{Feller1971}. The equivalence of the first two expressions for $c^=$ follows by considering for each $n$ the time-reversed walk $(S^*_k)_{0\le k\le n}$, where $S^*_k = S_n - S_{n-k}$, which has the same law as $(S_k)_{0\le k\le n}$ (see \cite[Section XII.2]{Feller1971}). Finally, the last expression for $c^=$ follows from setting $s=1$ in Lemma 2 in \cite[Section XVIII.3]{Feller1971}.
\end{proof}

From the ``duality lemma'' \cite[Section XII.2]{Feller1971}, we have the following equivalent representation of $\mu$:
\begin{equation}
\label{eq:mu_duality}
\mu(dx) = \sum_{n\ge0} \mathbb{P}(|S_n| \in dx,\ S_k < 0,\,1\le k\le n).
\end{equation}
The other measures have analogous representations with the ``$<$'' above replaced by, respectively, $\le$, $>$, $\ge$.

We are interested in the random walk killed when it enters the negative half-line. Define two functions $R,\bar R:\mathbb{R}_+\to\mathbb{R}_+$ by
\begin{align}
R(x) &= \mu([0,x]) \label{eq:Rdef}\\
\bar R(x) &= \begin{cases}
\mu^=([0,x)), & x>0\\
1, & x = 0
\end{cases} \label{eq:Rbardef}
\end{align}
The function $R$ is also called the \emph{renewal function} associated to the strictly descending ladder heights of the random walk $(S_n)_{n\ge0}$.
The following lemma is originally due to Tanaka~\cite{Tan1989}, a good reference is \cite[Lemma 4.2, p78]{Kersting2017}
\begin{lemma}
\label{lem:R}
The functions $R$ and $\bar R$ are harmonic for the random walk killed when it enters $(-\infty,0)$ and $(-\infty,0]$, respectively, i.e. for all $x\ge0$,
\begin{align*}
R(x) &= \mathbb{E}[R(x+S_1)\boldsymbol{1}_{x+S_1\ge 0}]\\
\bar R(x) &= \mathbb{E}[\bar R(x+S_1)\boldsymbol{1}_{x+S_1> 0}]
\end{align*}
\end{lemma}

Finally, define two Green operators by
\begin{align*}
G f(x) &= \sum_{n=0}^\infty \mathbb{E}[f(x+S_n)\boldsymbol{1}_{x+S_k\ge 0,\,1\le k\le n}]\\
\bar G f(x) &= \sum_{n=0}^\infty \mathbb{E}[f(x+S_n)\boldsymbol{1}_{x+S_k > 0,\,1\le k\le n}].
\end{align*}
For $x=0$, Equation~\eqref{eq:mu_duality} (with $\hat\mu$ instead of $\mu$) gives
\begin{align*}
Gf(0) = \int_{[0,\infty)} \hat \mu^=(dz) f(z) = c^= \int_{[0,\infty)} \hat \mu(dz) f(z) = c^= \bar Gf(0)
\end{align*}
In general, the Green operators have the following expressions:
\begin{theorem}
\label{th:green}
For every measurable, non-negative function $f:\mathbb{R}_+\to\mathbb{R}_+$, we have
\begin{align*}
G f(x) &= c^= \int_{[0,x]\times [0,\infty)} (\mu\otimes \hat \mu)(dy,dz) f(x-y+z)\quad\forall x\ge 0\\
\bar G f(x) &= c^= \int_{[0,x)\times [0,\infty)} (\mu\otimes \hat \mu)(dy,dz) f(x-y+z)\quad \forall x>0.
\end{align*}
\end{theorem}

We were not able to find this result in the literature in this generality. For random walks on the integers, it was proven by Spitzer~\cite[P19.3, p209]{Spitzer1976}. The general result can, with some effort\footnote{The effort consists in identifying the mesures $\mu$ and $\hat \mu$ appearing there with the renewal measures associated to the compound Poisson process: this can be done for example by inspecting and adapting the treatment in \cite[Chapter VI]{Bertoin1996}.}, be deduced from a corresponding result for L\'evy processes~\cite{Tanaka2004}, using the fact that the Green operators and the renewal measures defined above are equal to the corresponding ones for the compound Poisson process associated to the random walk. However, all of these proofs make use of Sparre~Andersen's identities (see e.g. \cite[Section 4.2]{Kersting2017}) or their continuous analogue. Instead, we give here a direct proof, based on Spitzer's original idea\footnote{In fact, close inspection of Spitzer's proof shows that his use of Sparre~Andersen's identities can be easily avoided and is purely due to his choice of notation. This results in the proof we give here.}.

\begin{proof}[Proof of Theorem~\ref{th:green}]
~\\
Define
\[
M_n = \min_{0\le k\le n} S_k,\quad n=0,1,\ldots
\]
The following identity is well-known and can be easily obtained by decomposing the random walk at the first time it hits $M_n$ (see e.g. \cite[(2.2)]{Pecherskii1969} or the proof of \cite[Theorem 4.4]{Kersting2017}): for every $\lambda,\mu > 0$, we have
\begin{align*}
\sum_{n=0}^\infty \mathbb{E}[e^{\lambda M_n - \mu(S_n - M_n)}] = 
\left(\sum_{n=0}^\infty \mathbb{E}[e^{\lambda S_n}\boldsymbol{1}_{S_k<0,\,1\le k\le n}] \right)
\left(\sum_{n=0}^\infty \mathbb{E}[e^{-\mu S_n}\boldsymbol{1}_{S_k\ge 0,\,1\le k\le n}] \right).
\end{align*}
Equivalently, by \eqref{eq:mu_duality} and Lemma~\ref{lem:c=}, for every measurable, non-negative function $g$, we have
\begin{align*}
\sum_{n=0}^\infty \mathbb{E}[g(M_n,S_n-M_n)] = c^= \int_{[0,\infty)^2} (\mu\otimes \hat \mu)(dy,dz) g(-y,z)
\end{align*}
The expressions for $Gf(x)$ and $\bar Gf(x)$ then follow by setting in the above equation $g(y,z) = f(x+y+z)\boldsymbol{1}_{y \ge -x}$ and $g(y,z) = f(x+y+z)\boldsymbol{1}_{y > -x}$, respectively.
\end{proof}


\section{Laplace transform criterion for convergence in probability}
\label{appLaplace}
\begin{lemma}\label{lem:AnnexB}
Consider a probability space $(\Omega,\AA,\mathbb{P})$. Let $(\GG_n)_{n\ge0}$ be a filtration. Let $Y, Y_{0}, Y_1,...$ be non-negative r.v.s defined on that probability space. Suppose $Y$ is mesurable with respect to $\GG_{\infty}=\bigvee_{k\ge0}\GG_{k}$. Suppose that for every $\lambda\in\mathbb{Q}^+=\mathbb{Q}\cap(0,\infty)$,
\begin{equation}\label{eq:laplacecondconv}
\mathbb{E}\left[e^{-\lambda Y_n}\Big{|}\GG_{n}\right] \underset{n\to\infty}{\longrightarrow}e^{-\lambda Y}\quad\text{almost surely}.
\end{equation}
Then $Y_n$ converges in probability to $Y$ as $n\to\infty$.
\end{lemma}

\begin{proof}[Proof of Lemma~\ref{lem:AnnexB}]
~\\
The proof proceeds in three steps.

\emph{First step:} Fix a representative $\mu_n$ of the law of $Y_n$ conditioned on $\GG_n$. We may see $\mu_n$ as a random measure on the compact space $[0,\infty]$. Hence, the sequence $(\mu_n(\omega))_{n\ge0}$ is tight for every $\omega\in\Omega$. We wish to show that $\mu_n$ weakly converges to $\dg_Y$ almost surely as $n\to\infty$.

By \eqref{eq:laplacecondconv}, there exists an event $E\subset \Omega$ of probability 1, such that for every $\omega\in E$, for every $\lambda\in\mathbb{Q}$, $\lambda>0$, with $e^{-\infty} \coloneqq 0$,
\[
\lim_{n\to\infty} \int_{[0,\infty]} e^{-\lambda x}\,\mu_n(\omega)(dx) = e^{-\lambda Y}.
\]
In particular, for every $\omega\in E$, every subsequential limit $\mu$ of $\mu_n(\omega)$ satisfies for every $\lambda\in\mathbb{Q}$, $\lambda>0$,
\[
\int_{[0,\infty]} e^{-\lambda x}\,\mu(dx) = e^{-\lambda Y},
\]
and by continuity, this can be extended to every $\lambda > 0$. Since the Laplace transform characterizes a probability measure on $[0,\infty]$, this shows that $\mu = \dg_Y$. We have thus shown:
\begin{equation*}
\mu_{n}\underset{n\to\infty}{\longrightarrow}\dg_{Y}\quad\text{a.s.}
\end{equation*}

\emph{Second step:} We wish to show that the law of $Y$ conditioned on $\GG_n$ weakly converges to $\dg_Y$ as well, almost surely as $n\to\infty$. Note that, for every $\lg\in\mathbb{Q}^+$, $\left(\mathbb{E}[e^{-\lg Y}|\GG_n]\right)_n$ is a bounded $\GG_n$-martingale, so it converges almost surely towards $\mathbb{E}[e^{-\lg Y}|\GG_\infty]=e^{-\lg Y}$. We can then use the same argument as above to obtain that the law of $Y$ conditioned on $\GG_n$ converges a.s. towards $\dg_{Y}$.

\emph{Third step:} Denote by $\bar{\mu}_n$ the law of the pair $(Y_n,Y)$ conditioned on $\GG_n$. By the steps 1 and 2 above and Slutsky's lemma, $\bar{\mu}_n$ converges a.s. towards $\dg_{(Y,Y)}$. Hence, putting $g(x,y)=|x-y|\wedge 1$, which is a bounded and continuous function, we get
\begin{equation*}
\mathbb{E}\left[|Y_n-Y|\wedge 1\big{|}\GG_n\right]=\int g(x,y)d\bar{\mu}_{n}(x,y)\underset{n\to\infty}{\longrightarrow}0\quad\text{a.s.}
\end{equation*}
Using the dominated convergence theorem, we conclude that $\mathbb{E}[|Y_n-Y|\wedge1]$ converges to 0, i.e. that $Y_n$ converges in probability to $Y$. This concludes the proof.
\end{proof}

\section{A Tauberian-type lemma}
\label{appTauberian}

\begin{lemma}\label{lem:tauberian}
Let $Y$ be a positive random variable, $\rho$ a regularly varying function at $\infty$ of index strictly greater than $-1$ and define $\forall y\ge 0,\varphi(y)=\mathbb{E}[e^{-y}Y\wedge 1]$. Then the following statements are equivalent:
\begin{align*}
\int_{0}^{\infty}\varphi(y)\rho(y)\,dy&<\infty\\
\mathbb{E}\left[(\log_+Y)\rho(\log_+Y)\right]&<\infty
\end{align*}
\end{lemma}

\begin{proof}
~\\
For all $s>0$, put $l(s)=\rho(\log(s))$, which is slowly varying by Proposition 1.5.7 (ii) of \cite{Bingham1987}.
Following the notations from \cite{Bingham1987}, we put $\forall s\ge0,f_0(s)=1-\mathbb{E}[\exp(-sY)]$. Using the inequalities 
\[
 \forall x\ge0,\quad(1-e^{-1}) (x\wedge 1) \le 1-e^{-x} \le x\wedge 1,
\]
we get
\begin{equation*}
\forall y\ge0,\quad (1-e^{-1})\varphi(y)\le f_0(e^{-y})\le\varphi(y).
\end{equation*}
Using this, we see that
\begin{equation*}
\int_{0}^{\infty}\varphi(y)\rho(y)\,dy<\infty\Leftrightarrow\int_{0}^{\infty}f_{0}(e^{-y})l(e^{y})\,dy<\infty.
\end{equation*}
Changing variables, we have
\[
 \int_{0}^{\infty}f_{0}(e^{-y})l(e^{y})\,dy = \int_0^1 f_0(s)l\left(\frac 1 s\right)\frac{ds}{s}
\]

Using a Tauberian theorem from Bingham and Doney (see Theorem~8.1.8 in \cite{Bingham1987} with parameters $n=0$, $\beta=0$), we have
\begin{equation*}
\int_{0}^{\infty}f_{0}(e^{-y})l(e^{y})dy<\infty\Leftrightarrow\mathbb{E}\left[\int_{1}^{Y\vee 1}l(t)\frac{dt}{t}\right]<\infty.
\end{equation*}
Changing again variables, we have for every $y\ge0$,
\[
 \int_{1}^{y\vee 1}l(t)\frac{dt}{t} = \int_0^{\log_+ y} \rho(z)\,dz.
\]
Since $\rho$ is regularly varying with some index $\nu > -1$, we have the following asymptotic (see for example Theorem~1.5.11 from \cite{Bingham1987})
\[
 \int_0^{\log_+ y} \rho(z)\,dz \sim \frac 1 {\nu+1} (\log_+ y)\rho(\log_+ y),\quad y\to\infty.
\]
The two previous displays give,
\begin{equation*}
\mathbb{E}\left[\int_{1}^{Y\vee 1}l(t)\frac{dt}{t}\right]<\infty\Leftrightarrow\mathbb{E}\left[(\log_+Y)\rho(\log_+Y)\right] < \infty.
\end{equation*}
Collecting the above identities yields the statement of the lemma.
\end{proof}



\begin{thebibliography}{10}

\bibitem{Aidekon2013}
Elie A{\"{i}}d{\'{e}}kon, \emph{{Convergence in law of the minimum of a
  branching random walk}}, Annals of Probability \textbf{41} (2013), no.~3A,
  1362--1426.

\bibitem{Aidekon2014}
Elie Aid{\'{e}}kon and Zhan Shi, \emph{{The Seneta-Heyde scaling for the
  branching random walk}}, Annals of Probability \textbf{42} (2014), no.~3,
  959--993.

\bibitem{Aru2017}
Juhan Aru, Ellen Powell, and Avelio Sep{\'{u}}lveda, \emph{{Liouville measure
  as a multiplicative cascade via level sets of the Gaussian free field}},
  arXiv:1701.05872 (2017).

\bibitem{Bertoin1996}
Jean Bertoin, \emph{{L{\'{e}}vy processes}}, Cambridge Tracts in Mathematics,
  vol. 121, Cambridge University Press, Cambridge, 1996.

\bibitem{Biggins1977}
J.~D. Biggins, \emph{{Martingale convergence in the branching random walk}},
  Journal of Applied Probability \textbf{14} (1977), no.~01, 25--37.

\bibitem{Biggins2004}
J.~D. Biggins and A.~E. Kyprianou, \emph{{Measure change in multitype
  branching}}, Advances in Applied Probability \textbf{36} (2004), no.~2,
  544--581.

\bibitem{Bingham1987}
N.~H. Bingham, C.~M. Goldie, and J.~L. Teugels, \emph{{Regular Variation}},
  Encyclopedia of mathematics and its applications, vol.~27, Cambridge Univ.
  Press, Cambridge, 1987.

\bibitem{Bolthausen1976}
Erwin Bolthausen, \emph{{On a Functional Central Limit Theorem for Random Walks
  Conditioned to Stay Positive}}, The Annals of Probability \textbf{4} (1976),
  no.~3, 480--485.

\bibitem{BovierBook}
Anton Bovier, \emph{{Gaussian processes on trees}}, Cambridge Studies in
  Advanced Mathematics, vol. 163, Cambridge University Press, Cambridge, 2017.

\bibitem{Caravenna2008}
Francesco Caravenna and Lo{\"{i}}c Chaumont, \emph{{Invariance principles for
  random walks conditioned to stay positive}}, Annales de l'institut Henri
  Poincare (B) Probability and Statistics \textbf{44} (2008), no.~1, 170--190.

\bibitem{Chen2015}
Xinxin Chen, \emph{{A necessary and sufficient condition for the nontrivial
  limit of the derivative martingale in a branching random walk}}, Advances in
  Applied Probability \textbf{47} (2015), no.~3, 741--760.

\bibitem{Feller1971}
William Feller, \emph{{An introduction to probability theory and its
  applications. Vol II.}}, second ed., John Wiley \& Sons Inc., New York, 1971.

\bibitem{He2016}
Hui He, Jingning Liu, and Mei Zhang, \emph{{On Seneta-Heyde Scaling for a
  stable branching random walk}}, Advances in Applied Probability \textbf{50}
  (2018), 565--599.

\bibitem{Iglehart1974}
D.L. Iglehart, \emph{{Functional central limit theorems for random walks
  conditioned to stay positive}}, The Annals of Probability \textbf{2} (1974),
  no.~4, 608--619.

\bibitem{Kersting2017}
G{\"{o}}tz Kersting and Vladimir Vatutin, \emph{{Discrete Time Branching
  Processes in Random Environment. Volume 1}}, ISTE Ltd and John Wiley \& Sons,
  Inc., 2017.

\bibitem{Kesten1966}
H.~Kesten and B.~P. Stigum, \emph{{A limit theorem for multidimensional
  Galton--Watson processes}}, The Annals of Mathematical Statistics \textbf{37}
  (1966), no.~5, 1211--1223.

\bibitem{Kozlov1976}
M.~V. Kozlov, \emph{{On the asymptotic behavior of the probability of
  non-extinction for critical branching processes in a random environment}},
  Theory of probability and its applications \textbf{21} (1976), no.~4, 2037.

\bibitem{Kyprianou2015a}
Andreas~E. Kyprianou and Thomas Madaule, \emph{{The Seneta-Heyde scaling for
  homogeneous fragmentations}}, arXiv:1507.01559 (2015).

\bibitem{Lyons1998}
Russell Lyons, \emph{{A Simple Path to Biggins' Martingale Convergence for
  Branching Random Walk}}, Classical and Modern Branching Processes (K.B.
  Athreya and Peter Jagers, eds.), The IMA Volumes in Mathematics and its
  Applications, vol 84, Springer, New York, NY, 1997, pp.~217--221.

\bibitem{Lyons1995}
Russell Lyons, Robin Pemantle, and Yuval Peres, \emph{{Conceptual Proofs of
  LlogL Criteria For Mean Behavior of Branching Processes}}, The Annals of
  Probability \textbf{23} (1995), no.~3, 1125--1138.

\bibitem{Madaule2016}
Thomas Madaule, \emph{{First order transition for the branching random walk at
  the critical parameter}}, Stochastic Processes and their Applications
  \textbf{126} (2016), no.~2, 470--502.

\bibitem{slowdown}
Pascal Maillard and Ofer Zeitouni, \emph{{Slowdown in branching Brownian motion
  with inhomogeneous variance}}, Annales de l'Institut Henri Poincar{\'{e}},
  Probabilit{\'{e}}s et Statistiques \textbf{52} (2016), no.~3, 1144--1160.

\bibitem{Pecherskii1969}
E.~A. Pecherskii and B.~A. Rogozin, \emph{{On Joint Distributions of Random
  Variables Associated with Fluctuations of a Process with Independent
  Increments}}, Theory of Probability and its Applications \textbf{14} (1969),
  no.~3, 410--423.

\bibitem{RhodesVargasReview}
R{\'{e}}mi Rhodes and Vincent Vargas, \emph{{Gaussian multiplicative chaos and
  applications: A review}}, Probability Surveys \textbf{11} (2014), 315--392.

\bibitem{Rogozin1964}
BA~Rogozin, \emph{{On the Distribution of the First Jump}}, Theory of
  Probability and Its Applications \textbf{9} (1964), no.~3, 450--466.

\bibitem{ShiLectureNotes}
Zhan Shi, \emph{{Branching random walks}}, Lecture Notes in Mathematics.
  {\'{E}}cole d'{\'{E}}t{\'{e}} de Probabilit{\'{e}}s de Saint-Flour XLII --
  2012, vol. 2151, Springer, Cham, 2015.

\bibitem{Spitzer1976}
Frank Spitzer, \emph{{Principles of random walk}}, second ed., Graduate Texts
  in Mathematics, vol. 34, Springer-Verlag, New York - Heidelberg - Berlin,
  1976.

\bibitem{Tan1989}
Hiroshi Tanaka, \emph{{Time reversal of random walks in one dimension.}}, Tokyo
  J. Math. \textbf{12} (1989), no.~1, 159--174.

\bibitem{Tanaka2004}
\bysame, \emph{{L{\'{e}}vy processes conditioned to stay positive and
  diffusions in random environments}}, Stochastic analysis on large scale
  interacting systems, Adv. Stud. Pure Math., vol.~39, Math. Soc. Japan, Tokyo,
  2004, pp.~355--376.

\bibitem{ZeitouniLNBRW}
Ofer Zeitouni, \emph{{Branching random walks and {G}aussian fields}},
  Probability and statistical physics in {S}t. {P}etersburg, Proc. Sympos. Pure
  Math., vol.~91, Amer. Math. Soc., Providence, RI, 2016, pp.~437--471.

\end{thebibliography}

\providecommand{\bysame}{\leavevmode\hbox to3em{\hrulefill}\thinspace}
\providecommand{\MR}{\relax\ifhmode\unskip\space\fi MR }
\providecommand{\MRhref}[2]{%
  \href{http://www.ams.org/mathscinet-getitem?mr=#1}{#2}
}
\providecommand{\href}[2]{#2}


\ACKNO{We warmly thank Jean Bertoin, Ron Doney and Vladimir Vatutin for helpful discussions on random walks. Part of the work was done at the Centre de Recherches Math\'ematiques (CRM) Montr\'eal, which we thank for its hospitality. We also thank the anonymous referees whose detailed comments led to an improvement of the presentation.}


\end{document}